\RequirePackage[leqno]{amsmath}
\documentclass[11pt]{llncs}
\usepackage{bookmark}
\makeatletter
\providecommand*{\toclevel@titlech}{0}
\edef\toclevel@authorch{\the\numexpr\toclevel@titlech+1}
\makeatother

\usepackage[leqno]{amsmath}
\usepackage{amsfonts,amssymb}
\usepackage[sort&compress,comma,numbers,sectionbib]{natbib}
\usepackage{t1enc}
\usepackage{mathptmx}
\usepackage[utf8]{inputenc}
\usepackage{xcolor}
\usepackage{hyperref}
\usepackage{comment}
\usepackage{graphicx}
\usepackage{enumerate}
\usepackage{helvet}         % selects Helvetica as sans-serif font
\usepackage{courier}        % selects Courier as typewriter font
\usepackage{type1cm}        % activate if the above 3 fonts are

\usepackage{geometry}

\numberwithin{equation}{section}

\newcommand{\Int}{\operatorname{Int}}
\newcommand{\two}{\ensuremath{\mathbf{2}}}
\newcommand{\0}{\ensuremath{\mathbf{0}}}
\newcommand{\1}{\ensuremath{\mathbf{1}}}
\newcommand{\2}{\ensuremath{\mathbf{2}}}

\newcommand{\B}{\ensuremath{\mathfrak{B}}}
\newcommand{\A}{\ensuremath{\mathfrak{A}}}
\newcommand{\K}{\ensuremath{\mathbf{K}}}
\newcommand{\X}{\ensuremath{\mathcal{X}}}
\newcommand{\F}{\ensuremath{\mathcal{F}}}
\newcommand{\Eq}{\ensuremath{\mathbf{Eq}}}
\newcommand{\ds}{\oplus} %{\overset{\cdot}{+}}

\newcommand{\symdiff}{\triangle}
\renewcommand{\c}[1]{\ensuremath{-#1}}

\renewcommand{\phi}{\varphi}

\newcommand{\ult}{\operatorname{Ult}}
\newcommand{\mb}{\ensuremath{\mathcal{M}(B)}}
\newcommand{\cB}{\overline{B}}

\newcommand{\Cm}{\ensuremath{\operatorname{{\mathsf{Cm}}}}}   % Complex algebra
\newcommand{\Cf}{\ensuremath{\operatorname{{\mathbb{C}\mathsf{st}}}}}   % Atom structure
\newcommand{\ub}{\ensuremath{\operatorname{{ub}}}}

\newcommand{\df}{:=}%{\ensuremath{\overset{\mathrm{df}}{=}}}
\newcommand{\aright}{``$\Implies$'': \ }
\newcommand{\aleft}{``$\Leftarrow$'': \ }
\newcommand{\rel}{\operatorname{\mathrm{Rel}}}

\newcommand{\klam}[1]{\ensuremath{\langle #1 \rangle}}
\newcommand{\set}[1]{\ensuremath{\{#1\}}}
\newcommand{\z}{\emptyset}

\newcommand{\tand}{\text{ and }}
\newcommand{\tor}{\text{ or }}
\newcommand{\timplies}{\text{ implies }}
\newcommand{\tiff}{if and only if \ }

\newcommand{\Iff}{\Longleftrightarrow}
\newcommand{\Implies}{\Rightarrow}

\newcommand{\conv}[1]{#1\ensuremath{~\breve{}~}}

\newcommand{\poss}[1]{\ensuremath{\klam{ #1 }}}
\newcommand{\wlg}{w.l.o.g.\ }
\newcommand{\card}[1]{\vert #1 \vert}

\newcommand{\da}{\downarrow}

\newcommand{\Kt}{\ensuremath{\mathbf{K^{\sim}}}}
\newcommand{\wmia}{\ensuremath{\mathsf{wMIA}}}
\newcommand{\dda}{\ensuremath{\mathsf{DDA}}}
\newcommand{\kmpa}{\ensuremath{\mathsf{KDDA}}}
\newcommand{\moa}{\ensuremath{\operatorname{\mathsf{MOA}}}}
\newcommand{\dec}{\ensuremath{\operatorname{Dp}}}
\newcommand{\pfend}{\qed} % {\hfill{$\blacksquare$}}

\parskip=\medskipamount
\numberwithin{equation}{section}

\title{On the semilattice of modal operators and \\ decompositions of the discriminator}

\author{Ivo D{\"u}ntsch\inst{1}${}^{,}$\inst{2}${}^{,}$\thanks{Ivo D{\"u}ntsch gratefully acknowledges support by the National Science Fund of Bulgaria, contract DN02/15/19.12.2016.} \and Wojciech Dzik\inst{3} \and Ewa Or{\l}owska\inst{4}}
\institute{
Dept of Computer Science, {Brock University} \\
{St Catharines, ON, L2S 3A1, Canada} \and  School of Mathematics and Computer Science, Fujian Normal University, \\ Fuzhou, Fujian, China
\\
\href{mailto:duentsch@brocku.ca}{duentsch@brocku.ca}, \ \href{mailto:ivo@duentsch.net}{ivo@duentsch.net},
   \and Institute of Mathematics, University of Silesia\\ Katowice, Poland \\ \href{mailto:wojciech.dzik@us.edu.pl}{wojciech.dzik@us.edu.pl}, \
\href{mailto:wdzik@wdik.pl}{wdzik@wdzik.pl}
\and National Institute of Telecommunications \\
Szachowa 1,
04--894, Warszawa, Poland \\
\href{mailto:orlowska@itl.waw.pl}{orlowska@itl.waw.pl}
}

\begin{document}

\maketitle

\pagestyle{plain}

\makebox[\textwidth][r]{
\begin{minipage}[t]{0.48\textwidth}
\vspace{0mm}
\emph{Dedicated to our friends Hajnal and Istvan with respect and gratitude for long lasting inspiration}
\end{minipage}
}
\begin{abstract}
\noindent We investigate the join semilattice of modal operators on a Boolean algebra $B$. Furthermore, we consider pairs $\klam{f,g}$ of modal operators whose supremum is the unary discriminator on $B$, and study the associated bi--modal algebras.
\end{abstract}


\section{Introduction}

Boolean algebras with operators were introduced by \citet{jt51} in connection with their investigations into  relation algebras. It was observed much later that the simplest case of modal algebras, that is, expansions of Boolean algebras with a single unary normal operation which preserves finite joins could serve as algebraic semantics for the minimal modal logic $\K$ and its extensions. It is straightforward to see that the set $\mb $ of modal operators on a Boolean algebra $B$ can be made into a bounded join semilattice; here the smallest element $f^\0$ is the constant mapping $f \equiv 0$, and the largest element $f^\1$ is the unary discriminator.

In this paper, we study the lattice theoretic properties of $\mb$, in particular, the existence and form of dual pseudocomplements. If $B$ is a complete Boolean algebra, the problems concerning $\mb $ are solved. In particular,  $\mb $ is dually pseudocomplemented if and only if $B$ is complete.

Generalizing the notion of pseudocomplement, we consider pairs $\klam{f,g}$ of modal operators (\emph{companions}) whose join in $\mb $ is the discriminator $f^\1$. Such pairs lead to bimodal algebras $\klam{B,f,g}$ which we call \emph{discriminator decomposition algebras}. These algebras give rise to a study of bimodal logics in which the modalities $ f, g$  are connected  by the discriminator decomposition condition. Such logics are not included in the main stream of research on multimodal logics.

We observe that the equational class generated by such algebras is equipollent to the class of algebraic models of the logic \Kt\ which generalizes both $\K$ and its complementary counterpart $\K^{\mathbf{*}}$ \cite{gpt87}.

In the final part of the paper we address the question when $f$ has a proper companion, i.e. a companion $g$ with $g \neq f^\1$, and give answers for several classes of modal algebras. It turns out that the existence of a proper companion can be expressed in a 1st order language, and is related to both the structure of $B$ and the modal operator $f$.  In particular, we conclude that this property is not a global property of the logic, but depends on the model, for instance, whether $B$ has an atom or not.  Connections to some classes of modal algebras are given, and examples are provided throughout.

\section{Notation and first definitions}

We regard an ordinal as the set of its predecessors, and cardinals as initial ordinals; $\omega$ is the first infinite ordinal  \cite{mos06}.   As in our context no generality is lost, we shall tacitly assume that a class $\A$ of algebras is closed under isomorphic copies. If no confusion can arise, we will refer to an algebra simply by its base set. The equational class generated by $\A$ is denoted by $\Eq(\A)$.

The \emph{ternary discriminator function} on an algebra $A$ is a function $t: A^3 \to A$ defined by
\begin{gather*}
t(a,b,c) =
\begin{cases}
a, &\text{if } a \neq b, \\
c, &\text{if } a = b.
\end{cases}
\end{gather*}
A ternary term $t(x,y,z)$ which represents the discriminator function is called a \emph{ternary discriminator for $A$}. If $\A$ is a class of algebras with a common discriminator term, then $\Eq(\A)$ is called a \emph{discriminator variety}. Discriminator varieties have very strong universal algebraic properties, see e.g.  \cite{wer78} or \cite{ajn91}.

 If $B$ is a Boolean algebra, this can be simplified:  A mapping $d: B \to B$ is a \emph{unary discriminator function}, if
\begin{gather}
d(a) =
\begin{cases}
0, &\text{if } a =0, \\
1, &\text{otherwise}.
\end{cases}
\end{gather}
A unary term $d(x)$ which represents the unary discriminator function on $B$ is called a \emph{unary discriminator for $B$}. The next observation is well known:
\begin{lemma}\label{lemNOdisc}
A Boolean algebra $B$ has a unary discriminator \tiff it has a ternary discriminator.
\end{lemma}
A \emph{frame} is a pair $\klam{X,R}$ where $X$ is a set and $R$ a binary relation on $X$. The identity relation on $X$ is denoted by $1'_X$, or just by $1'$ if $X$ is understood; $V_X$ (or just $V$) denotes the universal relation. If $x \in X$, then $R(x) \df \set{y: xRy}$ is the \emph{range of $x$} (with respect to $R$).

%$R$ is called a \emph{right ideal relation}, if $R \comp V \subseteq R$ \cite{tarski_relalg}; here, $\comp$ denotes relational composition.

Suppose that $\klam{P, \leq}$ is a partially ordered set and $Q \subseteq P$. Then $\da Q \df \set{x: (\exists y)[y \in Q \tand x \leq y]}$ is the \emph{downset generated by $Q$}. If $Q = \set{x}$, we just write $\da x$ if no confusion can arise. If $P$ has a smallest element $0$, then $Q^+ \df Q \setminus \set{0}$, otherwise, $Q^+ \df Q$. $Q$ is called \emph{dense (in $P$)} if for every $y \in P^+$ there is some $x \in Q^+$ such that $x \leq y$. $\ub(Q)$ is the set of all upper bounds of $Q$.

In the sequel, a semilattice is assumed to be a join semilattice. Suppose that $\klam{S, \lor, 1}$ is an upwardly bounded semilattice. A \emph{dual annihilator of $x \in S$} is some $y \in S$ such that $x \lor y = 1$. Such $y$ is \emph{proper} if $y \neq 1$. $x$ is called \emph{dually dense} if its only dual annihilator is $1$. If $x \in S$ has a smallest annihilator, this element is called the \emph{dual pseudocomplement of $x$}, denoted by $x^\bot$. An element $ x \in S$ is called \emph{open} if $x = y^{\bot}$ for some $y \in S$. If each $x \in S$ has a dual pseudocomplement, the structure $\klam{S,\lor, {}^\bot, 1}$ is called a \emph{dually pseudocomplemented semilattice}. It is well known that the class of dually pseudocomplemented semilattices is equational, see e.g. \cite[p 104]{blyth05}.

\begin{lemma}\label{lemNOoba}
\cite[Theorem 1]{frink62} If $\klam{S, \lor, {}^\bot,1}$ is a dually pseudocomplemented semilattice, then the set $O(S)$ of open elements of $S$ is a $1$ -- subsemilattice of $S$ and a Boolean algebra with $x \land_{O(S)} y \df (x^\bot \lor y^\bot)^\bot$.
\end{lemma}

\subsection{Boolean algebras}

Throughout, $\klam{B, +, \cdot, -, 0,1}$ is a nontrivial Boolean algebra (BA), usually only referred to by its universe $B$. \two\ is the BA with universe $\set{0,1}$, and $\cB$ is the completion of $B$. The set of atoms of $B$ is denoted by $At(B)$, and the set of ultrafilters of $B$ is denoted by $\ult(B)$. The mapping $h: B \to \ult(B)$ with $h(x) = \set{F \in \ult(B): x \in F}$ is the Stone embedding.

We write $a = a_0 \ds \ldots \ds a_n$, if $a = a_0 + \ldots + a_n$, and the $a_i$ are nonzero and pairwise disjoint. The symmetric difference $x \cdot \c{y} + \c{x} \cdot y$ of $x,y \in B$ is denoted by $x \symdiff y$.

\begin{comment}
If $y \in B$ and $k \in \set{-1,1}$, we  set
%
\begin{gather*}
y^{[k]} =
\begin{cases}
y, &\text{if } k = 1, \\
-y, &\text{if } k = -1.
\end{cases}
\end{gather*}
\end{comment}
$B$ is called a \emph{finite--cofinite} algebra (FC--algebra), if every element of $B \setminus \2$ is a finite sum of atoms or the complement of such an element.  If $B$ is an FC--algebra, $\kappa$ a cardinal, and $\card{B} = \kappa$, then $B$ is isomorphic to the BA $FC(\kappa)$ which is generated by the one element subsets of $\kappa$. If $\gamma \in \kappa$, we let $F_\gamma$ be the ultrafilter of $FC(\kappa)$ generated by $\set{\gamma}$, and $F_\kappa$ be the ultrafilter of cofinite sets.

\begin{comment}
If $M \subseteq \ult(B), \ x \in B$, we say that \emph{$M$ admits $x$}, if $x \in \bigcap M$, i.e. if $M \subseteq h(x)$.
\end{comment}

If $M$ is dense in $B$ and $x \in B^+$, then $x = \sum\set{y \in M: y \leq x}$. Moreover, there is a pairwise disjoint family $M' \subseteq M$ such that $x = \sum M'$ \cite[Lemma 4.9.]{kop89}.

Recall some facts about Boolean interval algebras: Let $L$ be a linear order with smallest element $0_m$. Suppose that $\infty$ is a symbol not in $L$, and set $L^\infty \df L \cup\set{\infty}$ with $x \lneq \infty$ for all $x \in L$. An \emph{interval of $L$} is a set of the form $[s,t) = \set{u \in L: s \leq u \lneq t}$, where $s,t \in L^\infty$.  $IntAlg(L)$ is the collection of all finite unions of intervals
\begin{gather}\label{int}
[x^0_0, x^1_0) \cup [x^0_1, x^1_1) \cup \ldots \cup [x^0_{t(x)}, x^1_{t(x)}),
\end{gather}
together with the empty set. It is well known that $IntAlg(L)$ is a Boolean algebra \cite[p.10]{kop89}, called \emph{the interval algebra of $L$}.
Each nonzero $x \in IntAlg(L)$ can be written in the form \eqref{int} in such a way that $x^i_j \in L^+$, $x^0_j < x^1_j < x^0_{j+1}$; note that the intervals $[x^0_j, x^1_j)$ are pairwise disjoint. The representation of $x$ in this form is unique \cite[p. 242]{kop89}, and we call it the {\em standard representation}. For each $x \in IntAlg(L)^+$, we let
\begin{align*}
\rel(x) &\df \set{x^0_j: j \le t(x)} \cup \set{x^1_j: j \le t(x)} \\
\intertext{be the set of {\em relevant points} of $x$, and}
\Int(x) &\df \set{[x^0_j, x^1_j): j \leq t(x)}
\end{align*}
be the set of \emph{relevant intervals of $x$}.

For unexplained notation and concepts in the area of universal algebra the reader is invited to consult \cite{bs_ua}, and for Boolean algebras we refer the reader to \cite{kop89}.

\section{Modal algebras}

An \emph{operator} on $B$ is a mapping $B \to B$; note that this is more general than the terminology of \cite{jt51}. If $f$ is an operator on $B$, then its \emph{dual (operator)} $f^\partial$ is defined by $f^\partial (a) \df - f(-a)$. We also set $f^\ast(a) \df -f(a)$, and $f_\ast(a) \df f(-a)$; clearly, $f^* = (f^\partial)_\ast$.

A mapping $f: B \to B$ $f$ is called \emph{normal} if $f(0) = 0$, and \emph{additive} if $f(x + y) = f(x) + f(y)$ for all $x,y \in B$. A \emph{modal operator} $f$ on $B$ is a normal and additive mapping. In this case, $\klam{B,f}$ is called a \emph{modal algebra}. The class of modal algebras is denoted by $\moa$. It may be remarked that \moa\ is not locally finite, indeed, a modal algebra need not have a finite subalgebra; an example can be found in \cite[p 251]{jon93}.

A modal operator $f$ on $B$ is called \emph{completely additive}, or simply \emph{complete}, if for every $M \subseteq B$ such that $\sum M$ exists, $\sum\set{f(x): x \in M}$ also exists and is equal to $f(\sum M)$.

An ideal $I$ of $B$ is called \emph{trivial}, if $I = \set{0}$, \emph{proper} if $I \neq B$, and  \emph{closed}, if $a \in I$ implies $f(a) \in I$. It is well known that there is a one--one correspondence between closed ideals and congruences on $\klam{B,f}$, see e.g. \cite[\S 3]{blok80a}, and therefore, we will call closed ideals also \emph{congruence ideals}.

A modal operator $f$ is called a \emph{closure operator} or \emph{S4 operator}, if it satisfies
\begin{enumerate}[{Cl}$_1$]
\item $x \leq f(x)$,
\item $f(f(x)) = f(x)$.
\end{enumerate}
In this case, $\klam{B,f}$ is called a \emph{closure algebra} or an \emph{S4 algebra}.

If $\X \df \klam{X,R}$ is a frame, we define a mapping $\poss{R}: 2^X \to 2^X$ by $\poss{R}(Y) = \set{x: R(x) \cap Y \neq \z}$. In modal logic, $\poss{R}$ is the (interpretation of) the \emph{diamond operator} $\Diamond$ in the frame \X. The structure $\klam{2^X, \poss{R}}$ is called the \emph{complex algebra of $\X$}, denoted by $\Cm(\X)$.
\begin{theorem}\label{thmNOframealg} \cite[Theorem 3.9]{jt51}
\begin{enumerate}
\item $\Cm(\X)$ is  complete and atomic, and $\poss{R}$  is a completely additive normal operator.
\item If $\B = \klam{B,f}$ is a complete atomic Boolean algebra and $f$ is completely additive, then $\B$ is isomorphic to some $\Cm(\klam{X,R})$. \pfend
\end{enumerate}
\end{theorem}
The \emph{canonical structure} of a modal algebra $\B \df \klam{B,f}$ is the frame $\Cf(\B) \df \klam{\ult(B), R_f}$, where
\begin{gather}\label{canrel}
FR_fG \Iff f[G] \subseteq F \text{ i.e. } \set{f(a): a \in G} \subseteq F.
\end{gather}
$R_f$ is called \emph{the canonical relation of $f$}. If $a,b \in At(B)$, and $F_a, F_b$ are the principal ultrafilters generated by $a$, respectively, $b$, then
\begin{gather}\label{RfAt}
F_a R_f F_b \Iff a \leq f(b).
\end{gather}

Observe that $R_f$ is the universal relation on $\ult(B)$ \tiff $f = f^\1$, and $R_f$ is the empty relation \tiff $f = f^\0$.

The following representation theorem is seminal for algebraic semantics of modal logics:.
\begin{theorem}\label{thmNOrep} \cite[Theorem 3.10]{jt51}
Let $\B \df \klam{B,f} \in \moa$, and $h: B \to 2^{\ult(B)}$ be defined by $h(a) = \set{F \in \ult(B): a \in F}$. Then, $h$ is a MOA embedding into $\klam{2^{\ult(B)}, \poss{R_f}} ( = \Cm\Cf(\B))$; in particular, $h(f(a)) = \poss{R_f}h(a)$.
\end{theorem}
The algebra $\Cm\Cf(\B)$ is called the \emph{canonical extension of $\B$},  denoted by  $\B^\sigma$. It is well known that $\B \cong \B^\sigma$ \tiff $B$ is finite. We will sometimes assume that \wlg $\klam{B,f}$ is a subalgebra of $\B^\sigma$. In this case, $\poss{R_f}$ is denoted by $f^+$. We shall usually just write $B^\sigma$ instead of $\B^\sigma$.

 For a  history of and an introduction to Boolean algebras with operators, the reader is invited to consult \cite{jon93}.

We shall use several axioms of modal logics in their algebraic form and the corresponding property of their canonical relation:
\begin{enumerate}
\renewcommand{\theenumi}{K}
\item\label{K} $f(0) = 0$ and $f(x + y) = f(x) + f(y)$. \hfill{$R_f$ is a binary relation.}
\renewcommand{\theenumi}{T}
\item\label{T} $x \leq f(x)$. \hfill{$R_f$ is reflexive.}
\renewcommand{\theenumi}{4}
\item\label{4} $f(f(x)) \leq f(x)$. \hfill{$R_f$ is transitive.}
\renewcommand{\theenumi}{B}
\item\label{B} $f(f^\partial(x)) \leq x$. \hfill{$R_f$ is symmetric.}
\end{enumerate}

Conjunctions of axioms are usually written in juxtaposition of the axioms; for example,
KT means a modal logic based on the axioms \ref{K} and \ref{T}. Some abbreviations are common:
\begin{center}
KT4 $\mapsto$ S4, \quad KT4B $\mapsto$ S5.
\end{center}

The next result provides a convenient criterion for a modal algebra to be subdirectly irreducible:

\begin{theorem} (Rautenberg's criterion \cite[p. 155]{rau80})
Let $\B = \klam{B,f}$ be a modal algebra. Then, $\B$ is subdirectly irreducible \tiff
\begin{gather}\label{raut1}
(\exists a \neq 1)(\forall b \neq 1)(\exists n \in \omega) b \cdot f^\partial(b) \cdot (f^{\partial})^2(b) \cdot \ldots \cdot (f^{\partial})^n(b) \leq a.
\end{gather}
If \B\ is a K4 algebra, then \eqref{raut1} is equivalent to the statement
\begin{align}
&(\exists a \neq 1)(\forall b \neq 1)b \cdot f^\partial (b) \leq a, \label{raut2}
\end{align}
and if \B\ is an S4 algebra, i.e. if $f$ is a closure operator, then \eqref{raut1} is equivalent to
\begin{align}
&(\exists a \neq 1)(\forall b \neq 1) f^\partial (b) \leq a. \label{raut3}
\end{align}
\end{theorem}

\section{The semilattice of modal operators}

 Let $\mb $ be the set of all modal operators on a fixed Boolean algebra $B$.  For $f,g \in \mb $ set $(f \lor g)(x) \df f(x) + g(x)$. Furthermore, let $f^\0 :\equiv 0$, and
 \begin{gather*}
 f^\1(x) \df
 \begin{cases}
 0, &\text{if } x = 0, \\
 1, &\text{otherwise.}
 \end{cases}
 \end{gather*}
 Note that $f^\1$ is the unary discriminator on $B$. In modal logics $f^\1$ is known as the \emph{universal modality} \cite{gpt87,gp92}.

The following observation which, is straightforward to prove, is the basis for the considerations in this section:
 \begin{theorem}
 $\klam{\mb , \lor, f^\0, f^\1}$ is a bounded (join) semilattice.
 \end{theorem}
We remark in passing that the structure $\klam{\mb , \lor, f^\0, f^\1, \circ, 1'}$ is a bounded idempotent semiring where $\circ$ is composition of functions and $1'$ is the identity.

For each $x \in B$ define the relativization of $f^\1$ to $\da x$ by

\begin{gather*}
f_x(y) \df
\begin{cases}
0, &\text{if } y =0, \\
x, &\text{otherwise.}
\end{cases}
\end{gather*}
Clearly, each $f_x$ is a modal operator on $B$.

\begin{theorem}\label{thmNOcompsl}
$\mb $ is a complete semilattice \tiff $B$ is complete.
\end{theorem}
\begin{proof}
\aright Suppose that $\z \neq M \subseteq B$.  Let $f \df \bigvee\set{f_x: x \in M}$, and choose some $y \in B^+$. Our aim is to show that $f(y)$ is the least upper bound of $M$. Since $f_x(y) = x$ and $f_x \leq f$, it is clear that $f(y)$ is an upper bound of $M$. Next, let $z$ be an upper bound of $M$. Then, $x \leq z$ for all $x \in M$, and therefore, $f_x \leq f_z$. It follows that $f \leq f_z$, hence, $f(y) \leq f_z(y)  = z$, the latter by definition of $f_z$.

\aleft For $\z \neq Q \subseteq \mb $ set $ \bigvee Q(a) \df \sum\set{f(a): f \in Q}$ for each $a \in B$. In this case it is in fact a complete lattice: If $\z \neq Q \subseteq \mb $, then $f^\0$ is a lower bound of $Q$, and thus, $\bigvee\set{f \in \mb : f \text{ is a lower bound of } Q}$ is well defined, and clearly, it is the greatest lower bound of $Q$.
\pfend\end{proof}

Our next topic in this section is the existence of dual pseudocomplements in $\mb $. We start with a characterization of pseudocomplements in $\mb $:

\begin{lemma}\label{thmNOprep}
Let $f \in \mb $. If $f$ has a dual pseudocomplement $f^\bot$, then $f^\bot(x) = \sum\set{-f(y): 0 \lneq y \leq x}$ for each $x \in B^+$. Conversely, for each $x \in B^+$, if ~ $\sum\set{-f(y): 0 \lneq y \leq x}$ exists, then  $f^\bot(x) = \sum\set{-f(y): 0 \lneq y \leq x}$.
\end{lemma}
\begin{proof}
Let $\overline{B}$ be the completion of $B$. If $M \subseteq B$, then $\sum_{\cB} M$ denotes the supremum of $M$ in $\cB$. Since $B$ is a dense subalgebra of $\cB$, all sums existing in $B$ coincide with the sums in $\cB$; in particular, if $\sum_{\cB} M = b \in B$, then $\sum_{\cB} M = \sum_B M$.

\aright Suppose that $f^\bot$ is a dual pseudocomplement of $f$, and let $x \in B^+$, $0 \lneq y \leq x$. Then, $y \leq x$ implies $f^\bot(y) \leq f^\bot(x)$.  Since $f(y) + f^\bot(y) = 1$, we have $-f(y) \leq f^\bot(y)$, and thus,  $-f(y) \leq f^\bot(x)$. It follows that $f^\bot(x)$ is an upper bound of $\set{-f(y): 0 \lneq y \leq x}$.

 Assume that $\sum_{\cB} \set{-f(y): 0 \lneq y \leq x} \lneq f^\bot(x)$ for some $x \in B^+$. Choose some $a_x \in B$ such that
 \begin{gather*}
 \sum\nolimits_{\cB} \set{-f(y): 0 \lneq y \leq x} \leq a_x \lneq f^\bot(x),
 \end{gather*}
 and define $f': B \to B$ by $f'(0) \df 0$, and
\begin{gather*}
f'(z) \df
\begin{cases}
 1, &\text{if } z \not\leq x, \\
a_x, & \text{if } 0 \neq z \leq x.
\end{cases}
\end{gather*}
Then, $f' \in \mb $, and $f(z) + f'(z) = 1$, if $z \not\leq x$. If $0 \neq z \leq x$, then
\begin{xalignat*}{2}
f(z) + f'(z) &= f(z) + a_x, \\
&\geq f(z) + \sum\nolimits_{\cB} \set{-f(y): 0 \lneq y \leq x}, \\
&\geq f(z) + - f(z), && \text{since } 0 \neq z \leq x, \\
&= 1,
\end{xalignat*}

Since $f^\bot$ is the dual pseudocomplement of $f$, we have $f^\bot \leq f'$, in particular, $f^\bot(x) \leq f'(x) = a_x$. This contradicts $a_x \lneq f^\bot(x)$. Therefore, $\sum\nolimits_{B} \set{-f(y): 0 \lneq y \leq x}$ exists and is equal to $f^\bot(x)$ for each $x \in B^+$.

\aleft Set $f'(0) \df 0$, and $f'(x) \df \sum\set{-f(y): 0 \lneq y \leq x}$. Let $x,z \in B^+$.  If $x \leq x'$, then $\set{-f(y): 0 \lneq y \leq x} \subseteq \set{-f(y): 0 \lneq y \leq x'}$, and thus, $f'(x) \leq f'(x'$. It follows that $f'$ is isotone, and $f'(x) + f'(z) \leq f'(x+z)$.

Conversely,
\begin{align*}
f'(x+z) &= \sum\set{-f(y): 0 \lneq y \leq x+z}, \\
&=\sum\set{-f(y \cdot x + y \cdot z): y \leq x + z}, \\
&=\sum\set{-(f(y \cdot x) + f(y \cdot z)): y \leq x + z}, \\
&\leq \sum\set{-f(y): 0 \lneq y \leq x} + \sum\set{-f(y): 0 \lneq y \leq z}, \\
&= f'(x) + f'(z).
\end{align*}
Thus, $f' \in \mb $. Furthermore,
\begin{gather}\label{d1}
f(x) + f'(x) = f(x) + \sum\set{-f(y): 0 \lneq y \leq x} \geq f(x) + -f(x) = 1,
\end{gather}
and therefore, $f \lor f' = f^\1$.

Next, suppose that $f \lor g = f^\1$ for some $g \in \mb $. Let  $x \neq 0$ and assume that $f'(x) \not\leq g(x)$, i.e. $f'(x) \cdot - g(x) \neq 0$. Then,
\begin{gather*}
0 \neq  \underbrace{\sum\set{-f(y): 0 \lneq y \leq x}}_{f'(x)} \cdot - g(x) = \sum\set{-f(y) \cdot -g(x): 0 \lneq y \leq x}.
\end{gather*}
Thus, there is some $0 \lneq y \leq x$ such that $-f(y) \cdot -g(x) \neq 0$, i.e. $f(x) + g(x) \neq 1$. This  contradicts our hypothesis $f \lor g = f^\1$, and it follows that $f^\bot$ is the dual pseudocomplement of $f$.
\pfend\end{proof}

\begin{theorem}\label{thmNOdpc}
$\mb $ is dually pseudocomplemented \tiff $B$ is complete.
\end{theorem}
\begin{proof}
\aright Suppose that $\mb $ is dually pseudocomplemented, and let $M \subseteq B^+$. Let $F$ be the filter of $B$ generated by $M$. Then, $\prod_{\cB} F = \prod_{\cB} M$: Since $M \subseteq F$, any lower bound of $F$ is a lower bound of $M$. Conversely, let $x$ be a lower bound of $M$ and $y \in F$. Then, there are $q_0, \ldots, q_n \in M$ such that $q_0 \cdot \ldots \cdot q_n \leq y$, and therefore, $x \leq y$. Thus, we may assume that $M$ is a filter of $B$. Let $f$ be the identity. Then, by Lemma \ref{thmNOprep},
\begin{gather*}
f^\bot(1) = \sum\nolimits_{B}\set{-f(y): 0 \lneq y} =  \sum\nolimits_{B}\set{-y: 0 \lneq y} = - \prod\nolimits_B M.
\end{gather*}
Hence, $-\prod_B M$ exists and thus, $B$ is complete.

\aleft If $B$ is complete, then $\sum\set{-f(y): 0 \lneq y \leq x}$ exists for each $x \in B^+$, and the mapping defined by
\begin{gather*}
f^\bot(x) \df
\begin{cases}
0, &\text{if } x = 0, \\
\sum\set{-f(y): 0 \lneq y \leq x}, &\text{if } x \neq 0
\end{cases}
\end{gather*}
is the dual pseudocomplement of $f$ by Lemma \ref{thmNOprep}.
\pfend\end{proof}
In particular, $\mb $ is dually pseudocomplemented, if $B$ is finite.

Let $M_c(B)$ be the set of completely additive normal operators of $B$. It is obvious that $M_c$ is a sub--semilattice of $\mb $. If $B$ is complete, we can say more:
\begin{theorem}
Suppose that $B$ is complete.
\begin{enumerate}
\item\label{ps1} $M_c(B)$ is complete.
\item\label{ps3} $g = f^\bot$ for some $f \in \mb $ \tiff $g$ is completely additive.
\item\label{ps4} $M_c(B)$ is a Boolean algebra, and $f^{\bot\bot}$ is the largest element of $M_c(B)$ below $f$ for every $f \in \mb $.
\end{enumerate}
\end{theorem}
\begin{proof}
\ref{ps1}. Let $f,g \in M_c(B)$ and $\z \neq M \subseteq B$; since $B$ is complete, $\sum M$ exists. We first show that $f \lor g$ is complete, i.e. that $(f \lor g)(\sum M) = \sum\set{(f \lor g)(a): a \in M}$. Consider
\begin{xalignat*}{2}
(f \lor g)\left(\sum M \right) &= f\left(\sum M\right) + g\left(\sum M\right), &\text{by definition of $\lor$}, \\
&= \sum\set{f(a): a \in M} + \sum\set{g(a): a \in M},  &\text{$f$ and $g$ are complete}, \\
&= \sum\set{f(a) + g(a): a \in M}, &\text{infinite distrib. of $+$}, \\
&= \sum\set{(f \lor g)(a): a \in M}, &\text{by definition of $\lor$}.
\end{xalignat*}
It is straightforward to extend this over infinite joins of complete modal operators.

\ref{ps3}. \aright $f^\bot$ is completely additive: Suppose that $\z \neq M \subseteq B^+$, and $\sum M$ exists. The aim is to show that $\sum\set{f^\bot(m): m \in M}$ also exists and is equal to $f^\bot(\sum M)$. In the proof we shall use the de Morgan rules for infinite sums, see e.g. \cite[Lemma 1.33]{kop89}. Let $M = \set{m_i: i \in I}$. Now, $z \in \ub(\set{f^\bot(m_i): i \in I})$
\begin{xalignat*}{2}
& z \in \ub(\set{f^\bot(m_i): i \in I}) \\
 &\Iff z \in \ub( \set{\sum\set{-f(y): 0 \lneq y \leq m_i}: i \in I}), && \text{by Lemma \ref{thmNOprep}}, \\
&\Iff z \in \ub(\set{-f(y): 0 \lneq y \leq m_i, i \in I}), &&\text{by transitivity of } \geq, \\
&\Iff z \in \ub(\set{-f(y): 0 \lneq y, \ y = y \cdot m_i, i \in I}), \\
%&\Iff z \in \ub(\set{-f(y \cdot m_i): 0 \lneq y, \ y = y \cdot m_i, i \in I}), \\
&\Iff z \in \ub(\set{-f(t): 0 \lneq t \leq m_i}), &&\text{substituting $t$ for $y \cdot m_i$}, \\
&\Iff z \in \ub(\set{-f(t): 0 \lneq t \leq \sum M}),&& \text{$t \leq m_i$ implies $-f(m_i) \leq -f(t)$}\\
%\intertext{since $t \leq m_i$ implies $-f(m_i) \leq -f(t)$},
&\Iff z \in \ub f^\bot(\sum M).
\end{xalignat*}
Since $f^\bot(\sum M)$ exists, this implies that $\sum\set{f^\bot(m_i): i \in I}$ exists and is equal to $f^\bot(\sum M)$, and therefore, $f^\bot$ is complete.

\aleft Conversely, let $g \in M_c(B)$; we need to find some $f \in \mb $ such that $f^\bot = g$. The obvious candidate for $f$ is $g^\bot$ which exists, since $B$ is complete.

Since $g \lor g^\bot = f^\1$, we have $f^\bot = g^{\bot\bot} \leq g$, since $g^{\bot\bot}$ is the dual pseudocomplement of $g^\bot$. For $\geq$, first observe that
\begin{xalignat*}{2}
f^\bot(x) &= g^{\bot\bot}(x) = g^\bot(g^\bot(x)), \\
&= g^\bot(\sum\set{-g(y): 0 \lneq y \leq x}), &\text{by definition of $g^\bot$},\\
&= \sum\set{g^\bot(-g(y)): 0 \lneq y \leq x}, &\text{ since $g^\bot$ is compl. additive}, \\
&=  \sum\set{\sum\set{-g(z): 0\lneq z \leq -g(y)}: 0 \lneq y \leq x} &\text{by definition of $g^\bot$},\\
&= \sum\set{-g(z): 0\lneq z \leq -g(y) \tand 0 \lneq y \leq x}, &\text{by transitivity of $\geq$}.
\end{xalignat*}
Thus,
\begin{xalignat}{2}
&g(x) \leq f^\bot(x) \Iff g(x)\leq \sum\set{-g(z): 0\lneq z \leq -g(y) \tand 0 \lneq y \leq x}.
\end{xalignat}
Assume that $g(x) \not\leq f^\bot(x)$. Then, $g(x) \cdot -f^\bot(x) \neq 0$, i.e.
\begin{gather}\label{bb}
g(x) \cdot \prod\set{g(z): 0\lneq z \leq -g(y) \tand 0 \lneq y \leq x} \neq 0.
\end{gather}
If $g(x) = 1$, then $f^\bot(x) = g^{\bot\bot}(x) = 1$, and $g(x) \leq f^\bot(1)$, contradicting our assumption. If $g(x) \neq 1$, then $-g(x) \neq 0$, and therefore, \eqref{bb} implies that $g(x) \cdot -g(x) \neq 0$ by \eqref{bb}, a contradiction.

\ref{ps4}:  Since the open elements of $\mb $ are exactly the completely additive ones by \ref{ps3}., $M_c(B)$ is a Boolean algebra by Lemma \ref{lemNOoba}. The join $\lor_c$ in $M_c(B)$ coincides with $\lor$, and the meet $f \land_c g$ is given by $(f^\bot \lor g^\bot)^\bot$. Furthermore the assignment $f \mapsto f^{\bot\bot}$ is an interior operator by Theorem 2 of \cite{frink62}. This implies the second claim.
\pfend\end{proof}
The following observation comes as no surprise:
\begin{theorem}\label{ex:framealg}
Let $\X = \klam{X,R}$ be a frame; then, $\poss{R}^\bot = \poss{-R}$.
\end{theorem}
\begin{proof}
Since both $\poss{R}^\bot$ and $\poss{-R}$ are completely additive, and thus determined by their action on the singletons, it suffices to show that $\poss{R}^\bot(\set{y}) = \poss{-R}(\set{y})$ for every $y \in X$. Let $y \in X$. Then, by Lemma \ref{thmNOprep} and a simple computation,
\begin{xalignat*}{2}
&\poss{R}^\bot(\set{y}) = -\poss{R}(\set{y}) = \poss{-R}(\set{y}).
\end{xalignat*}
This proves the claim.
\pfend\end{proof}

Let \wlg $B \leq B^\sigma$, $f \lor g = f^\1$, and suppose that $b \in B$. Even though $f^{+\bot}(b)$ need not be equal to $g^+(b)$ (and not even be in $B$), it seems reasonable to ask whether $f^{+\bot}(b) \leq g^+(b)$ in $B^\sigma$.

\begin{lemma}\label{lemNOfg}
Let $f,g \in \mb $ such that $f + g = f^\1$, and  suppose that $F,G \in \ult(B)$. Then, $f[G] \not\subseteq F $ implies $g[G] \subseteq F$.
\end{lemma}
\begin{proof}
Let $b \in G$ such that $f(b) \not\in F$ and assume that there is some $c \in G$ such that $g(c) \not\in F$. Since $f(b) \not\in F$ we have $f(b \cdot c) \not\in F$, and since $g(c) \not\in F$ we have $g(b \cdot c) \not \in F$.  Thus, $f + g = f^\1$ implies $f(b \cdot c) \in F$, a contradiction.
\pfend\end{proof}

\begin{corollary}
Suppose that $B \leq B^\sigma$ and $f \lor g = f^\1$. Then, $f^{+\bot}(b) \leq g^+(b)$ for all $b \in B$.
\end{corollary}
\begin{proof}
Let $Y = h(a)$, $a \not \in \two$. We need to show that $\klam{-R_f}(Y) \subseteq \poss{R_g}(Y)$. Consider
\begin{align*}
\klam{-R_f}(Y) \subseteq \poss{R_g}(Y) &\Iff (\forall F)[(-R_f)(F) \cap Y \neq \z \Implies R_g(F) \cap Y \neq \z], \\
&\Iff (\forall F)[(\exists G)( a \in G \tand f[G] \not\subseteq F)\Implies R_g(F) \cap Y \neq \z], \\
&\Iff (\forall F,G)[a \in G \tand f[G] \not\subseteq F \Implies (\exists G')(a \in G' \tand g[G'] \subseteq F)].
\end{align*}
The claim now follows from Lemma \ref{lemNOfg} setting $G' \df G$.
\pfend\end{proof}

%\begin{comment}

\begin{example}\label{ex:jon2}
The first example (from \cite[p 251]{jon93}) exhibits a modal operator on the non--complete $FC(\omega)$ which has a dual pseudocomplement. Let $B \df FC(\omega)$ and let $f: B \to B$ be defined by $f(M) = \set{n+1: n \in M}$. Observe that $f^n(\omega) = \omega \setminus \set{0, \ldots n-1}$, and therefore, $\klam{B,f}$ does not have a finite subalgebra.

Let $g: B \to B$ be defined by
\begin{gather*}
g(M) =
\begin{cases}
\z, &\text{if } M = \z, \\
\omega \setminus \set{n+1}, &\text{if } M = \set{n}, \\
%\bigcup\set{g(\set{n_i}: n_i \in M)} , &\text{if } M = \set{n_0, \ldots,n_k}, \\
\omega, &\text{otherwise.}
\end{cases}
\end{gather*}
Clearly, $g$ is a modal operator. Suppose that $\z \neq M \subseteq \omega$. If $M$ is not an atom, then $g(M) = \omega$ by definition. If $M = \set{n}$, then $f(M) \cup g(M) = \set{n+1} \cup (\omega \setminus \set{n+1}) = \omega$. Thus, $f \lor g = f^\1$.

Let $f \lor h = f^\1$; our aim is to show that $g \leq h$. Since $f(\set{n}) \cup h(\set{n})= \set{n+1} \cup h(\set{n}) = \omega$, it follows that $h(\set{n}) = \omega \setminus \set{n+1}$ or $h(\set{n}) = \omega$. Thus, $g(\set{n} \subseteq h(\set{n})$. Let $\card{M} \geq 2$ and $n_0, n_1 \in M$. Then, $h(\set{n_0, n_1}) = h(\set{n_0}) \cup h(\set{n_1}) \supseteq \omega \setminus \set{n_0+1} \cup \omega \setminus \set{n_1+1} = \omega$. Hence, $g$ is the dual pseudocomplement of $f$.
\pfend\end{example}

Since $FC(\omega)$ is not complete,  $M(FC(\omega))$ is not dually pseudocomplemented. It is therefore instructive to give a concrete example of a modal operator on $FC(\omega)$ without dual pseudocomplement.

\begin{example}\label{ex:FC}
Let $B = FC(\omega)$, and define $f: B \to B$ by $f(\z) = \z$, and
\begin{gather*}
f(\set{n}) =
\begin{cases}
\set{0}, &\text{if } n = 0, \\
\omega \setminus \set{0}, &\text{if } n \neq 0  \text{ and $n$ is even}, \\
\omega \setminus \set{n}, &\text{if $n$ is odd,}
\end{cases}
\end{gather*}
and extend $f$ over $FC(\omega)$ by $f(M) = \bigcup\set{f(\set{n}): n \in M}$. Since every cofinite $M \subseteq \omega$ contains a positive even number $n$ and an odd number $m$, we note that
\begin{gather}\label{cof}
f(M) \supseteq f(\set{n}) \cup f(\set{m}) = \omega \setminus \set{0} \cup \omega \setminus \set{m} = \omega.
\end{gather}
Furthermore, if $g \in \mb $ and $n$ is odd, then
\begin{gather}\label{inodd}
f(\set{n}) \cup g(\set{n}) = \omega \timplies n \in g(\set{n}).
\end{gather}

Let $\F^0$ be the set of positive even numbers, and $\F^1$ be the set of odd numbers. For $0 \lneq i$ let $n_i$ be the i-th nonzero even number, $g_i(\z) \df \z$, and
\begin{xalignat}{2}
%g_i(\set{0}) &= \omega, \\
\label{g1} g_i(M) &\df \omega, &&\text{if $0 \in M$}, \\
\label{g2} g_i(\set{n} &\df \set{0}, &&\text{ if } n \in \F^0, \\
\label{g3} g_i(\set{n} &\df \set{n}, &&\text{ if } n \in \F^1, \\
\label{g5} g_i(M) &\df \omega \setminus \set{n_i}, &&\text{if $M$ is cofinite and }  0 \not\in M.
\end{xalignat}
We extend the $g_i$ additively over finite sets; then $g_i(M)$ is finite, if $M$ is finite.

Note that the $g_i$ only differ in how they handle a cofinite $M$ with $0 \not\in M$. Let $L, M \in B^+$. If, say, $0 \in M$, then $g_i(M) = \omega$, and $g_i(L \cup M) = \omega$, since $0 \in L \cup M$. Now,
\begin{gather*}
g_i(L \cup M) = \omega = g_i(M) \subseteq g_i(L) \cup g_i(M).
\end{gather*}
Thus, suppose that $0 \not\in L \cup M$. If $L$ is finite, then $g_i(L) \subseteq \set{n: n \in L \cap \F^1} \cup \set{0}$ by \eqref{g2} and \eqref{g3}. If both $L$ and $M$ are finite, then $g_i(L \cup M) = g_i(L) \cup g_i(M)$ by the definition of $g_i$.  If $M$ is cofinite, then $L \cup M$ is cofinite, and
\begin{gather*}
g_i(L) \subseteq \set{n: n \in L \cap \F^1} \cup \set{0} \subseteq \omega \setminus \set{n_i} = g_i(M).
\end{gather*}
Thus,
\begin{gather*}
g_i(L \cup M) = \omega\setminus \set{n_i} = g_i(M) = g_i(M) \cup g_i(L).
\end{gather*}
If both $L$ and $M$ are cofinite, then so is $L \cup M$, and $g_i(L) = g_i(M) = g_i(L \cup M) = \omega \setminus \set{n_i}$. Altogether, we have shown that $g \in \mb $.

Next, let $M \in B^+$. If $M$ is cofinite, then $f(M) = \omega$ by \eqref{cof}, and if $0 \in M$, then $g_i(M) = \omega$. Let $M$ be finite and $0 \not\in M$; it is enough to show that $f(\set{n}) \cup g_i(\set{n}) = \omega$ for $n \neq 0$. By definition,
\begin{gather*}
f(\set{n}) \cup g_i(\set{n} =
\begin{cases}
\omega \setminus \set{0} \cup \set{0} = \omega, &\text{if $n$ is even}, \\
\omega \setminus \set{n} \cup \set{n} = \omega, &\text{if $n$ is odd.}
\end{cases}
\end{gather*}
Hence, $f \lor g_i = f^\1$.

Assume that $g$ is a dual pseudocomplement of $f$. Then $g(\omega\setminus\set{0}) \leq g_i(\omega\setminus\set{0}) = \omega \setminus \set{n_i}$ for all $i \in \omega^+$, and thus, $g(\omega\setminus\set{0})$ contains no positive even numbers. However, $\omega\setminus\set{0}$ contains all odd numbers, and thus, $\F^1 \subseteq g(\omega\setminus\set{0})$ by \eqref{inodd}. It follows that $g(\omega \setminus \set{0}) \not\in B$, a contradiction. Thus, $f$ does not have a dual pseudocomplement.

It is instructive to consider the canonical extension $B^\sigma$ of $B$ with $f^+ \df \poss{R_f}$. Then, $f^{+\bot}$ exists since $B^\sigma$ is complete, and it is equal to $\poss{-R_f}$ by Theorem  \ref{ex:framealg}. Let $F_n$ be the principal ultrafilter of $FC(\omega)$ generated by $\set{n}$, and $U$ be the non--principal ultrafilter of cofinite sets; furthermore, let $h: B \to 2^{\ult(B)}$ be the Stone embedding. Then, for $M \in B$,
\begin{gather*}
h(M) =
\begin{cases}
\set{F_n: n \in M}, &\text{if $M$ is finite}, \\
\set{F_n: n \in M} \cup \set{U}, &\text{if $M$ is cofinite}.
\end{cases}
\end{gather*}
By \eqref{RfAt} and the definition of $f$,
\begin{gather*}
F_nR_fF_m \Iff
\begin{cases}
n = 0 \tand  (m = 0 \tor m \text{ odd}) \tor \\
n \text{ even and } m \neq 0, \tor \\
n \text{ odd and } n \neq m.
\end{cases}
\end{gather*}
Furthermore,
\begin{align*}
UR_f F_n &\Iff n \neq 0, \\
F_n R_f U &\Iff n \in \omega, \\
%\intertext{and}
UR_fU. &
\end{align*}
Therefore, keeping in mind that $\poss{R}(\set{F_m}) = \conv{R}(\set{F_m})$ we obtain
\begin{gather*}
\poss{R}(\set{F_m}) =
\begin{cases}
\set{0} \cup \set{F_n: n \text{ odd}}, &\text{if } m = 0, \\
\ult(B) &\text{if } m \in \F^0, \\
\ult(B) \setminus \set{F_m}&\text{if } m \in \F^1.
\end{cases}
\end{gather*}
This gives us
\begin{gather*}
\poss{-R}(\set{F_m}) =
\begin{cases}
 \set{F_n: n \in \F^0}, &\text{if } m = 0, \\
\z &\text{if } m \in \F^0, \\
 \set{F_m}&\text{if } m \in \F^1.
\end{cases}
\end{gather*}
This shows that $\poss{-R_f} \leq g_i$ for all $i \in \omega$. On the other hand, if we define
\begin{gather*}
p(\set{n}) =
\begin{cases}
\set{0}, &\text{if $n$ is even}, \\
\set{n}, &\text{if $n$ is odd},
\end{cases}
\end{gather*}
then $\klam{f,p}$ is a minimal pair, and $p \not\leq g_i$ for all $i \in \omega^+$.
%This shows that in Lemma \ref{lemNOminpair} below the assumption of completeness in one direction cannot be removed.
\pfend\end{example}

Our final example in this section exhibits a modal operator on the countable free Boolean algebra without a dual pseudocomplement.

\begin{example}\label{ex:free}
Let $B$ be the interval algebra of the rational unit interval $[0,1)_{\mathbb Q} $; we regard $B$ as a subalgebra of the real unit interval $[0,1)_{\mathbb R}$. It is well known that $B$ is the free Boolean algebra on countably many generators. In particular, $B$ is homogenous, and every nonempty infinite open interval is order isomorphic to every other nonempty infinite open interval.

Let $0 \lneq p \lneq 1$ be irrational, and $h: (0,1)_{\mathbb Q} \to (p, 1)_{\mathbb Q}$ be an order isomorphism.
Suppose that $x \in B^+$ and $[x^0_0, x^1_0) \cup [x^0_1, x^1_1) \cup \ldots \cup [x^0_{t(x)}, x^1_{t(x)})$ be its canonical representation. Now, set $f(0) \df 0$, and, for $x \in B^+$,
\begin{gather*}
f(x) \df [0, h(x^1_{t(x)})).
\end{gather*}
Let $y = [y^0_0, y^1_0) \cup [y^0_1, y^1_1) \cup \ldots \cup [y^0_{t(y)}, y^1_{t(y)})$; then,
\begin{xalignat*}{2}
f(x) + f(y) &= [0, h(x^1_{t(x)})) \cup [0, h(y^1_{t(y)})) \\
&= [0, \max\set{h(x^1_{t(x)}), h(y^1_{t(y)})}, \\
&= [0, h(\max\set{x^1_{t(x)}, y^1_{t(y)}})), &&\text{since $h$ is an order isomorphism},\\
&= [0, h((x+y)^1_{t(x+y)})), &&\text{since }t(x+y) = \max\set{x^1_{t(x)}, y^1_{t(y)}}, \\
&= f(x + y).
\end{xalignat*}
Thus, $f \in \mb $. For each $x \in B^+$ let $M_x \df \set{-f(y): 0 \lneq y \leq x}$ and $M_x^* \df  \set{f(y): 0\lneq y \leq x}$; then, $M_x^* = \set{[0, h(y^1_{t(y)})): 0\lneq y \leq x}$. To abbreviate notation, let $C \df IntAlg([,0,1)_{\mathbb R})$. Considering that $h$ is an order isomorphism, we obtain
\begin{align}
\prod\nolimits_C M_x^* &= \prod\nolimits_C\set{[0, h(y^1_{t(y)})): 0\lneq y \leq x} = [0,p), \\
\intertext{and therefore}
\sum\nolimits_C M_x &= - \prod\nolimits_C\set{f(y): 0\lneq y \leq x} = - \prod\nolimits_C M_x^* =
[p,1). \label{ia2}
\end{align}
Next, let $f \lor f' = f^\1$, and $0 \lneq x \lneq 1$. Since $f'(x)$ is an upper bound of $M_x$ by Lemma   \ref{thmNOprep}, it follows from \eqref{ia2} that
\begin{gather}\label{p0}
[p, 1) \subseteq f'(x).
\end{gather}
For each $s \in (0,p) \cap \mathbb{Q}$, let $f_s(0) \df 0$, and $f_s(x) \df [s,1)$ for all $x \in B^+$. Then, $f_s \in \mb $, and
\begin{gather*}
f(x) \cup f_s(x) =  [0, h(x^1_{t(x)})) \cup [s,1) = [0,1),
\end{gather*}
the latter since $s \lneq p$. If $0 \lneq s \leq s' \lneq p$, then $f_{s'}(x) \leq f_s(x)$, and therefore,
\begin{gather}\label{p1}
\prod\nolimits_C\set{f_s(x): s \in (0,p) \cap \mathbb{Q}} = \prod\nolimits_C\set{[s,1): s \in (0,p) \cap \mathbb{Q}} = [p, 1).
\end{gather}

Assume that $f^\bot$ is a dual pseudocomplement of $f$. Then $[p,1) \subseteq f^\bot(x)$ for each $x \in B^+$ by \eqref{p0}; furthermore, $f^\bot(x) \leq f_s(x)$ for each $s \in (0,p) \cap \mathbb{Q}$, and thus, $f^\bot(x) \subseteq \prod\nolimits_C\set{f_s(x): s \in (0,p) \cap \mathbb{Q}}$, and \eqref{p1} implies that $f^\bot(x) \subseteq [p,1)$. Altogether, we obtain that $f^\bot(x) = [p,1)$, a contradiction.
\pfend\end{example}

As in Example \ref{ex:FC}, we see concretely that the pseudocomplement does not exist, because certain infinite products (or sums) do not exist in $B$ -- this time for an atomless BA. Lemma \ref{thmNOprep} tells us that this is to be expected, since $B$ is not complete. Furthermore, we observe that $f$ is a closure operator on the free countable Boolean algebra.

\section{Decomposing discriminators}

It is often the case that pairs of operators are considered which, taken together, have desirable structural properties. Examples of such pairs are Galois connections or residuated mappings. In \cite{dot_mixed} pairs of operators $\klam{f,g}$ were considered where $f$ is a modal operator, and $g$ is a sufficiency operator, i.e.
\begin{gather*}
g(0) = 1 \tand g(x + y) = g(x) \cdot g(y)
\end{gather*}
for all $x,y \in B$. A \emph{weak mixed algebra} (wMIA) is a structure $\klam{B,f,g}$ such that $f$ is a modal operator, $g$ is a sufficiency operator, and
\begin{gather}\label{wmia}
(\forall x)[x \neq 0 \timplies g(x) \leq f(x)].
\end{gather}
The class of wMIAs is denoted by \wmia. These algebras are intimately connected to algebraic models of the logic \Kt, which was introduced by \citet{gpt87}. It turns out that the discriminator decomposition algebras defined below are another way of describing weak MIAs.

Suppose that $f, g$ are modal operators on $B$, and consider the condition
\begin{gather}\label{fg}
(\forall x)[x \neq 0 \Implies f(x) + g(x) = 1.
\end{gather}
Clearly, $f \lor g$ is the unary discriminator. If a pair $\klam{f,g}$ of modal operators satisfies \eqref{fg} we call it a \emph{decomposing pair}, and $g$ a \emph{companion of $f$}. If $\klam{f,g}$ is a decomposing pair, then so is $\klam{g,f}$ owing to the commutativity of $+$. The set of all decomposing pairs is denoted by $\dec(B)$. For each $f \in \mb $ the pair $\klam{f,f^\1}$ is decomposing. If $f^\bot$ is the dual pseudocomplement of $f$, then $\klam{f,f^\bot}$ is a decomposing pair, and $f^\bot$ is the smallest $g \in \mb $ such that $\klam{f,g}$ is decomposing.

A \emph{discriminator decomposition algebra} (DDA) is a bi--modal algebra $\klam{B,f,g}$ such that $\klam{f,g}$ is a decomposing pair. If both $f$ and $g$ are proper, i.e. not equal to $f^\1$, $\klam{B,f,g}$ is called a \emph{proper DDA}. The class of DDAs is denoted by \dda; by \eqref{fg}, \dda\ is a discriminator class.  The relational counterpart of this situation are frames $\klam{X,R,S}$, where $R \cup S = X$; these are called \emph{generalized models of \Kt} in \cite{gpt87}.

\begin{theorem}\label{thmNOtranslate}
There is a bijective correspondence between the set of decomposing pairs and the set of pairs $\klam{f,g}$ such that $\klam{B,f,g}$ is a weak MIA.
\end{theorem}
\begin{proof}
Let $\klam{B,f,g}$ be a weak MIA, and set $f' \df g^\ast$. Then, $f'$ is a modal operator, and, for all $a \neq 0$,
\begin{gather*}
f(a) + f'(a) = f(a) + g^\ast(a) = f(a) + -g(a) = 1,
\end{gather*}
the latter since $g(a) \leq f(a)$. Clearly, the assignment $\klam{f,g} \mapsto \klam{f, g^*}$ is an injective mapping, and all that is left to show is that it is surjective. Thus, let $f,f'$ be modal operators such that $f(a) + f'(a) = 1$ for all $a \neq 0$, and set $g \df f'^\ast$. Clearly, $g$ is a sufficiency operator, and $f(a) + f'(a) = 1$ implies $g(a) = -f'(a) \leq f(a)$ for all $a \neq 0$.
\pfend\end{proof}
Thus, the classes \wmia\ and \dda\ are equipollent in the sense of \citet{tg87}. As in \dda\ we are dealing with just one kind of operator instead of the two kinds in \wmia, it is less complicated to work in \dda. It is also easier to apply results from the theory of modal algebras.

\begin{lemma}
\begin{enumerate}
\item \dda\ is closed under taking subalgebras,  homomorphic images and ultraproducts.
\item \dda\ is not closed under taking direct products, and thus, it is neither a variety nor a quasivariety.
\end{enumerate}
\end{lemma}
\begin{proof}
1. \dda\ is a universal class with a set of positive axioms, thus, it is closed under taking subalgebras, ultraproducts,  and homomorphic images, see e.g. \cite[Paper 5]{bs_ua}.

2. Since \dda\ is a discriminator class, each member is simple \cite[Theorem 2.2]{wer78}, and thus, \dda\ is not closed under direct products. For the rest, just note that each quasivariety is closed under direct products \cite[Theorem 2.25]{bs_ua}.
\pfend\end{proof}

The equational class generated by \wmia\ was described in \cite[Section 7]{dot_mixed}, and the results can be translated for \dda\ in a straightforward way. Given a bimodal algebra $\klam{B,f,g}$ let $u: B \to B$ be defined by
\begin{gather}\label{def:u}
u(x) \df f^\partial(x) \cdot g^\partial(x).
\end{gather}
$B$ is called a \emph{\Kt\ -- DDA}, if $u$ is an S5 possibility operator, i.e. if $u$ has the following properties:
\begin{align}
x &\le u(x), \label{u1} \\
u(u(x)) &\leq  u(x), \label{u2} \\
u(u^\partial(x))  &\leq x. \label{u3}
\end{align}
The class of \Kt\ -- DDAs is denoted by \kmpa; clearly, \kmpa\ is an equational class.

\begin{theorem}
$\Eq(\dda) = \kmpa$.
\end{theorem}
\begin{proof}
Taking into account Theorem \ref{thmNOtranslate}, the proof is a straightforward translation of \cite[Theorem 7.3]{dot_mixed}.
\pfend\end{proof}
The following result relating a decomposing pair to its canonical relations comes as no surprise:
\begin{theorem}\label{lemNOrel}
Suppose that $\klam{B,f,g}$ is a bimodal algebra. Then, $\klam{B,f,g} \in \dda$ \tiff $R_f \cup R_{g} = \ult(B)^2$.
\end{theorem}
\begin{proof}
\aright Let $\klam{F,G} \in \ult(B)^2$, and assume $\klam{F,G} \not\in R_f \cup R_{g}$. Then, $f[G] \not\subseteq F$ and  $g[G] \not\subseteq F$. Thus, there are $a,b \in G$ such that $f(a) \not\in F$ and $g(b) \not\in F$. Now, $a \cdot b \in G$, and $f(a \cdot b) \in F$ implies $f(a) \in F$, since $f$ is isotone and $F$ is a filter. Thus, $f(a \cdot b) \not\in F$, and, similarly, $g(a \cdot b) \not\in F$. Since $\klam{f,g}$ is a decomposing pair we have $f(a \cdot b) + g(a \cdot b) = 1 \in F$. Since $F$ is an ultrafilter, $f(a \cdot b) \in F$ or $g(a \cdot b) \in F$, a contradiction.

\aleft Let $a \in B^+$ and assume that $f(a) + g(a) \neq 1$. Then, there is an ultrafilter $F$ such that $f(a) + g(a) \not\in F$, i.e. $f(a) \not\in F$ and $g(a) \not\in F$. It follows that $f[G] \not\subseteq F$ and $g[G] \not\subseteq F$ which contradicts $R_f \cup R_{g} = \ult(B)^2$.
\pfend\end{proof}

\begin{theorem}
If $\klam{B,f,g} \in \dda$, then $\Cm\Cf(B) \in \dda$.
\end{theorem}
\begin{proof}
This follows from the syntactic form of \eqref{fg}, see. e.g. \cite[Theorem 4.2.1]{jon93}. A direct proof is as follows: Since $B \in \dda$, $R_f \cup R_{g} = \ult(B)^2$. We show that $\poss{R_f}(\set{U}) \cup \poss{R_{g}}(\set{U}) = \ult(B)$ for $U \in \ult(B)$; then this can be extended to all non--empty subsets of $\ult(B)$, since the operators are isotone. Assume that $\poss{R_f}(\set{U}) \cup \poss{R_{g}}(\set{U}) \neq \ult(B)$. Then, there is some ultrafilter $F$ of $B$ such that $F \not\in \poss{R_f}(\set{U})$ and $F \not \in \poss{R_{g}}(\set{U})$; in other words, $\klam{F,U} \not\in R_f$ and $\klam{F,U} \not\in R_{g}$. This contradicts $R_f \cup R_{g} = \ult(B)^2$.
\pfend\end{proof}

\section{Proper companions}

Let us order $\dec(B)$ by setting $\klam{f,g} \leq \klam{f',g'}$ if $f \leq f'$ and $g \leq g'$. The following observation is obvious:

\begin{lemma}\label{lemNOcomp}
If $g$ is a companion of $f$ and $g \leq g'$, then $g'$ is a companion of $f$.
\end{lemma}

\begin{theorem}\label{thmNOminpair}
Let $\klam{f,g} \in \dec(B)$. If $f$ and $g$ are dual pseudocomplements of each other, then $\klam{f,g}$ is a minimal pair. If $B$ is complete, then the converse also holds.
\end{theorem}
\begin{proof}
Since $f$ and $g$ are dual pseudocomplements of each other, we may suppose that $f = f^{\bot\bot}$ and $g = f^\bot$.
Let $\klam{f',g'} \in \dec(B)$, and $\klam{f',g'} \leq \klam{f^{\bot\bot}, f^\bot}$. Then, $f^{\bot\bot} \lor g' = f^\1$, since $f' \leq f^{\bot\bot}$. It follows that $f^\bot \leq g'$, since $f^\bot$ is the dual pseudocomplement of $f^{\bot\bot}$, hence, $f^\bot = g'$. Similarly we can show that $f^{\bot\bot} = f'$.

If $B$ is complete, then every $f \in \mb $ has a dual pseudocomplement, and the claim follows from the fact that the set $\set{\klam{f^{\bot\bot}, f^\bot}: f \in \mb }$ is dense in $\dec(B)$.
\pfend\end{proof}

In Example \ref{ex:FC}, $\klam{f,p}$ is a minimal pair, and $p \not\leq g_i$ for all $i \in \omega^+$. This shows that the assumption of completeness in the $\Leftarrow$ direction cannot be removed.

A companion $g$ of $f$ is called \emph{proper}, if $g \neq f^\1$. Note that this notion is not symmetric: If $f = f^\1$, then every $g \neq f^\1$ is a proper companion of $f$, but $f$ is not a proper companion of any $g$. A discriminating pair $\klam{f,g}$ is called \emph{proper} if both $f,g \neq f^\1$.

It turns out that the property of $f$ having a proper companion is a $\Sigma_1$ first order property, as the following result shows:

\begin{theorem}\label{thmNOpc}
$f$ has a proper companion \tiff
\begin{gather}\label{pc}
(\exists x)(\exists z)[x \neq 0 \tand z \neq 0 \tand (\forall y)(0 \lneq y \leq x \Implies z \leq f(y))].
\end{gather}
\end{theorem}
\begin{proof}
\aright Suppose that \eqref{pc} is not true, i.e.
\begin{gather}\label{nopc}
(\forall x)(\forall z)[x \neq 0 \tand (\forall y)(0 \lneq y \leq x \Implies z \leq f(y)) \Implies z = 0].
\end{gather}
For $x \neq 0$, set $S(x) \df \set{-f(y): 0 \lneq y \leq x}$. If $s$ is an upper bound of $S(x)$, then $-f(y) \leq s$ for all $0 \lneq y \leq x$, i.e. $-s \leq f(y)$. It now follows from \eqref{nopc} that $-s = 0$, i.e. $s = 1$, hence, $\sum S(x) = 1$ for all $x \in B^+$. By Lemma \ref{thmNOprep}, $f$ has a dual pseudocomplement $f^\bot$, and $f^\bot(x) =  \sum S(x) = 1$. Therefore, $f^\bot = f^\1$ which is the smallest companion of $f$.

\aleft Let $x,z$ witness \eqref{pc}; in particular, $x,z \neq 0$. We shall consider two cases:
\begin{enumerate}
\item $z = 1$: By the hypothesis, $(\forall y \neq 0)[y \leq x \Implies f(y) = 1]$. Set
\begin{gather}\label{gc1}
g(y) =
\begin{cases}
0, &\text{if } y = 0, \\
x, &\text{if } y \leq x, \\
1, &\text{otherwise.}
\end{cases}
\end{gather}
Clearly, $g \in \mb  \setminus \set{f^\1}$. Let $y \in B^+$. If $y \leq x$, then $z \leq f(y) = 1$, and if $y \not\leq x$, then $g(y) = 1$. It follows that $g$ is a proper companion of $f$.

\item $z \neq 1$: Define $g: B \to B$ by
\begin{gather}\label{gc2}
g(y) =
\begin{cases}
0, &\text{if } y =0, \\
-z, &\text{if } 0 \lneq y \leq x, \\
1, &\text{otherwise.}
\end{cases}
\end{gather}
Clearly, $g \in \mb $, and $g \neq f^\1$ since $-z \neq 1$. Let $y \in B^+$. If $y \leq x$, then $z \leq f(y)$ by the hypothesis, and $g(y) = -z$, which implies $f(y) + g(y) = 1$. If $y \not\leq x$, then $g(y) = 1$. Altogether, $g$ is a proper companion of $f$.
\end{enumerate}
This completes the proof.
\pfend\end{proof}

Since \eqref{pc} holds \tiff $\set{f(y): 0 \lneq y \leq x}$ has a nonzero lower bound for some $x \in B^+$, we obtain

\begin{corollary}\label{cor:prod}
The following statements are equivalent:
\begin{align}
&\text{$f$ has no proper companion.} \\
& (\forall x \in B^+) \prod\set{f(y): 0 \lneq y \leq x} = 0. \label{prodprop} \\
& (\forall u \in B \setminus \set{1}) \sum\set{f^\partial (y): u \leq y \lneq 1} = 1. \label{prodprop'}
\end{align}
\end{corollary}

Since $f(0) = 0$, the non--existence of a proper companion is in some sense an expression of continuity of $f$ at $0$. One may also interpret this as completeness of $f$ at $0$.

\begin{corollary}\label{cor:atom}
%\begin{enumerate}
%\item
If $a$ is an atom of $B$ with $f(a) \neq 0$, then $f$ has a proper companion.
%%\item If $B$ is atomic, then every nonzero $f \in \mb$ has a proper companion.
%\end{enumerate}
\end{corollary}
\begin{proof}
Set $x \df a$, and $z \df f(a)$; then, $x$ and $z$ are witnesses for \eqref{pc}.
%
%2. If $B$ is atomic and $f \neq f^\0$ is additive, then $f(a) \neq 0$ for some atom $a$ of $B$.
\pfend\end{proof}

\begin{theorem}\label{thmNOsiprop}
If $\B = \klam{B,f}$ is subdirectly irreducible, and $f$ is a closure operator, then $f$ has a proper companion.
\end{theorem}
\begin{proof}
Considering $-x$ and $-z$ we see that \eqref{pc} is equivalent to
\begin{gather}\label{pc'}
(\exists u \neq 1 )(\exists v \neq 1 ) (\forall t)(u \leq t \lneq 1 \Implies f^\partial (t) \leq v)].
\end{gather}
Since \B\ is subdirectly irreducible and $f$ is a closure operator, we can use Rautenberg's criterion in the form \eqref{raut3} to obtain some $a \in B$ such that $f^\partial(b) \leq a$ for all $b \neq 1$. Setting $u \df 0$ and $v:= a$ in \eqref{pc'} gives the desired result.
\pfend\end{proof}

%This can be applied to the interior algebras of certain topological spaces, see \cite{nr93}.

The following examples shed light on the connection between $f$ having a proper companion and $f[B]$ being dense.
%The first one exhibits an algebra where $f[B]$ is dense, and $f$ has no proper companion.

\begin{example}\label{ex:uf}
Suppose that $B$ is atomless and $I$ is a dense ideal of $B$. Define
\begin{gather*}
f(x) \df
\begin{cases}
1, &\text{if } x \not\in I, \\
x, &\text{otherwise.}
\end{cases}
\end{gather*}
Clearly, $f$ is a closure operator. Let $g$ be a companion of $f$. If $x \in I^+$, then, since $B$ is atomless,  there are $y, z \in I^+$ such that $y + z = x$ and $y \cdot z = 0$. By \eqref{fg}, $y + g(y) = z + g(z) = 1$, and therefore, $ -y \leq g(y)$ and $-z \leq g(z)$. Since $y \cdot z = 0$, we have $-y + -z = 1$, and it follows that
\begin{gather*}
1 = -y + -z \leq g(y) + g(z) = g(y + z) = g(x).
\end{gather*}
If $x \not\in I$, there is some $y \in I^+$ such that $y \leq x$, since $I$ is dense. Since $g(y) = 1$ and $g(y) \leq g(x)$, we have $g(x) = 1$. Hence, $g$ is not proper.
\pfend\end{example}

The next example destroys density of $f[B]$ while keeping part of the previous example. Recall that $B$ is called \emph{homogeneous} if $B \cong \da a$ for every $a \in B^+$ \cite[Definition 9.12.]{kop89}. In particular, every infinite free algebra is homogeneous.

\begin{example}\label{ex:notdense}
Suppose that $B$ is homogeneous, $0 \lneq a \lneq 1$, and that $\pi: \da a \to \da -a$ is an isomorphism. Define $f:B \to B$ by
\begin{gather*}
f(x) \df x + \pi(x \cdot a).
\end{gather*}
Then, $\da a \cap f[B] = \set{0}$, and thus, $f[B]$ is not dense in $B$. Since $f(0) = \pi(0) + 0 = 0$, we see that $f$ is normal: Let $x,y \in B^+$. Then,
\begin{align*}
f(x) + f(y) &= x + \pi(x \cdot a) + y + \pi(y \cdot a), \\
&= x + y + \pi(x \cdot a) + \pi(y \cdot a), \\
&= x + y + \pi(x \cdot a + y \cdot a), \\
&= x + y + \pi((x+y) \cdot a), \\
&= f(x+y).
\end{align*}
Clearly, $x \leq f(x)$. Furthermore,
\begin{align*}
\pi(f(x) \cdot a) &= \pi(\underbrace{(x + \pi(x \cdot a)}_{f(x)}) \cdot a) \\
&= \pi(x \cdot a + \pi(x \cdot a) \cdot a) \\
&= \pi(x \cdot a) + \pi(\pi(x \cdot a) \cdot a) \\
&= \pi(x \cdot a) + 0 = \pi(x \cdot a),
\end{align*}
and therefore,
\begin{xalignat*}{2}
f(f(x)) &= f(f(x) + \pi(f(x) \cdot a), \\
&= f(f(x) + \pi(x \cdot a)), \\
&= f(x + \pi(x \cdot a) + \pi(x \cdot a)), \\
&= f(x + \pi(x \cdot a)), \\
&= f(x) + f(\pi(x \cdot a)), && \text{since $f$ is additive}\\
&= x + \pi(x \cdot a) + \pi(x \cdot a), &&\text{since }\pi(x \cdot a) \cdot a = 0 \\
&= x + \pi(x \cdot a), \\
&= f(x).
\end{xalignat*}
Thus, $f$ is a closure operator.

Let $0 \lneq x$; by Corollary \ref{cor:prod} it is sufficient to show that $\prod\set{f(y): 0 \lneq y \leq x} = 0$. In what follows, we suppose that an infinite product $\prod M$ is taken in the completion $\overline{B}$ of $B$. If $\prod M_{\cB} = 0$, then $\prod M = 0$ in $B$.
\begin{align*}
\prod\nolimits_{\cB}\set{f(y): 0 \lneq y \leq x} &= \prod\nolimits_{\cB}\set{y + \pi(y \cdot a): 0 \lneq y \leq x} \\
&= \prod\nolimits_{\cB}\set{y: 0 \lneq y \leq x} + \prod\nolimits_{\cB}\set{\pi(y \cdot a): 0 \lneq y \leq x}, \\
&= 0 + 0,
\end{align*}
the latter since $B$ is atomless.
\pfend\end{example}

Our final example shows that density of $f[B]$ is in general not sufficient to show that $f$ has no proper companion.

\begin{example}\label{ex:densepc}
Let $B$ be homogeneous, and $0 \lneq a \lneq 1$. Furthermore, let $b,c \neq 0$ and $b \ds c = -a$. Since $\da a \cong \da b$ and $\da -a \cong \da c$, there are isomorphisms $\pi: \da b \to \da a$, and $\psi: \da c \to \da -a$. Furthermore,  $a \ds b \ds c = 1$ implies that $x = x \cdot a  \ds x \cdot b \ds x \cdot c$. Define $f: B \to B$ as follows:
\begin{gather*}
f(x) \df
\begin{cases}
%0, &\text{if } x = 0, \\
1, &\text{if } x \cdot a \neq 0, \\
\pi(x \cdot b) + \psi(x \cdot c). &\text{if } x \cdot a = 0.
%\pi(x), &\text{if } x \leq b, \\
%\psi(x), &\text{if } x \leq c, \\
%f(a \cdot x) + f(b \cdot x) + f(c \cdot x), &\text{otherwise}.
\end{cases}
\end{gather*}
Then, $f$ is well defined since $x = x \cdot a  \ds x \cdot b \ds x \cdot c$. Furthermore, $f(x) = \pi(x)$ if $x \leq b$, and $f(x) = \psi(x)$ if $x \leq c$.

Since $f(0) = \pi(0) = \psi(0) = 0$, $f$ is normal. Let $x,y \in B^+$. If, say, $x \cdot a \neq 0$, then $f(x) = 1$, and thus, $f(x) + f(y) = 1$. Since $(x + y) \cdot a \neq 0$, we also have $f(x + y) = 1$.

Let $x, y \leq -a$. Then,
\begin{align*}
f(x + y) &= \pi((x+y) \cdot b) + \psi((x+y) \cdot c), \\
&= \pi(x \cdot b + y \cdot b) + \psi(x \cdot c + y \cdot c), \\
&= \pi(x \cdot b) + \pi(y \cdot b) + \psi(x \cdot c) + \psi(y \cdot c), \\
&= \pi(x \cdot b) + \psi(x \cdot c) + \pi(y \cdot b) + \psi(y \cdot c), \\
&= f(x) + f(y).
\end{align*}
Thus, $f \in \mb $.

Let $z \in B^+$. If $z \cdot a \neq 0$, then $f(x) = \pi(x) = z \cdot a  \leq z$ for some $0 \lneq x \leq b$. If $z \cdot -a \neq 0$, then there is some $0 \lneq x \leq c$ such that $f(x) = \psi(x) = z \cdot -a \leq z$. Hence, $f[B]$ is dense in $B$.

Let $g: B \to B$ be defined by
\begin{gather*}
g(x) \df
\begin{cases}
0, &\text{if } x = 0, \\
a, &\text{if } 0 \lneq x \leq a, \\
1, &\text{otherwise.}
\end{cases}
\end{gather*}
Clearly, $g \in \mb  \setminus\set{f^\1}$. Let $x \neq 0$. If $x \leq a$, then $f(x) = 1$, and if $x \not\leq a$, then $g(x) = 1$. Hence, $g$ is a proper companion of $f$.
\pfend\end{example}

Thus, we observe that
\begin{enumerate}
\item There is some $\klam{B,f}$ such that $f[B]$ is dense and $f$ does not have a proper companion (Example \ref{ex:uf}).
\item There is some $\klam{B,f}$ such that $f$ is a closure operator, $f[B]$ is not dense and $f$ has no proper companion (Example \ref{ex:notdense}).
\item There is some $\klam{B,f}$ such that $f[B]$ is dense and $f$ has a proper companion (Example \ref{ex:densepc}).
\end{enumerate}
Therefore, the properties of $f[B]$ dense and $f$ having a proper companion are independent. However, if $f$ has additional properties, the situation is different, as we shall show below.

For a modal operator $f$ on $B$,  its $n-$th iteration is defined as usual: For $n \geq 1$, $x \in B$
\begin{align*}
f^1(x) &\df f(x), \\
f^{n+1}(x) &\df f(f^{n}(x)).
\end{align*}
 In analogy to the corresponding property of frame relations, we say that a modal operator $f$ is \emph{$n$--transitive} if, for all $x \in B$,
\begin{enumerate}
\renewcommand{\theenumi}{4$^{n}$}
\item $f^{n+1}(x) \leq f^{n}(x)$. \label{4n}
\end{enumerate}

\begin{theorem}\label{thmNOclopc}
Suppose that $B$ is atomless, and $f$ is $n$-transitive for some $n \geq 1$. If $f^n[B]$ is dense in $B$, then $f$ does not have a proper companion.
\end{theorem}

\begin{proof}
Assume that $f$ has a proper companion. By Theorem \ref{thmNOpc}, there are $x,z \in B^+$ such that $0 \lneq y \leq x$ implies $z \leq f(y)$. Let $t \in B$ such that $0 \lneq f^n(t) \leq x$; such $t$ exists, since $f^n[B]$ is dense by the hypothesis. By \eqref{pc}, $z \leq f(f^n(t))$, and thus, $z \leq f(f^n(t)) \leq  f^n(t)\leq x$, since $f$ is $n$-transitive. Note that this is not possible in Example \ref{ex:densepc}.

Again by density there is some $s \in B^+$ such that $0 \lneq f^n(s) \leq z$; since $B$ is atomless, we may suppose that $f^n(s) \lneq z$. But then, $0 \lneq s \leq x$ implies $z \leq f(f^n(s)) \leq f^n(s)$, a contradiction.
\pfend\end{proof}

Example \ref{ex:notdense} shows that the converse is not true, thus, density of $f[B]$ in $B$ is too strong a property to follow from not having a proper companion, even if $f$ is a closure operator.

The modal axiom $B: \Diamond \Box A \rightarrow A$ can be $n$-generalized in two ways, see e.g. \cite [Ch. 4.3, pp. 136f]{chellas80}:

\begin{quote}
\begin{enumerate}
\renewcommand{\theenumi}{B$^{n}$}
\item $\Diamond^n \Box^n A \rightarrow A$. \label{Bn}
\renewcommand{\theenumi}{B()$^{n}$}
\item $(\Diamond \Box)^n A \rightarrow A$. \label{B'n}
\end{enumerate}
\end{quote}

\begin{corollary}
The possibility operator of the Lindenbaum--Tarski algebra $\klam{B_\omega,f}$ of the logics $K4^nB^n$ and $K4^n B()^n$ does not have a proper companion for every $n \geq 1$.
\end{corollary}
\begin{proof}
Since $B_\omega$ is the countable free Boolean algebra, it is atomless. Furthermore, \ref{4n} says that $f^n$ is $n$--transitive, and  \ref{Bn} and \ref{B'n} imply that $f^n[B_\omega]$ is dense in $B_\omega$.
\pfend\end{proof}
%\begin{corollary}
%The possibility operator $f$ of the Lindenbaum--Tarski algebra $\klam{B_\omega,f}$ of the logic K4B with infinitely many variables does not have a proper companion.
%\end{corollary}

Therefore, the modal operator of the countable Lindenbaum--Tarski algebra of an axiomatic extension of K4B -- in particular, that of S5 --  does not have a proper companion. On the other hand, Corollary \ref{cor:atom} shows that any nontrivial modal operator on an algebra with at least one atom has a proper companion, in particular, that of a finite algebra with at least four elements. Therefore, we conclude that this property is not solely a property of the logic, but depends on the model.

\begin{comment}
\section*{Acknowledgements}

Ivo D{\"u}ntsch gratefully acknowledges support by the Bulgarian NFS contract DN02/15/19.12.2016.  Ewa Or{\l}owska gratefully acknowledges partial support from the National Science Centre project DEC-2011\-/02/A/HS1/00395.
\end{comment}

\renewcommand\bibname{References}

%\bibliographystyle{splncsnat}
%\bibliography{ddo}

\end{document}